\documentclass[a4paper, 11pt]{amsart}
\usepackage{amsmath, amssymb, mathrsfs}
\usepackage[mathscr]{eucal}
\usepackage{graphicx}
\usepackage{amsfonts,amscd}
\usepackage[mathscr]{eucal}
\usepackage[usenames]{color}
\usepackage[colorlinks=true,linkcolor=webgreen,filecolor=webbrown,citecolor=webgreen]{hyperref}
\definecolor{webgreen}{rgb}{0,.5,0}
\definecolor{webbrown}{rgb}{.6,0,0}
\newtheorem{theorem}{Theorem}
\newtheorem{corollary}{Corollary}
\newtheorem{lemma}[theorem]{Lemma}

\theoremstyle{definition}

\newtheorem{thm}{Theorem}

\newtheorem{remark}[thm]{Remark}

\def\Im{\operatorname{Im}}
\def\Re{\operatorname{Re}}

\def\exp{\operatorname{exp}}

\def\Li{\operatorname{Li}}
\def\O{\operatorname{O}}
\def\o{\operatorname{o}}
\def\log{\operatorname{log}}

\begin{document}
\title[Zeros of higher derivatives of Riemann zeta function]{Zeros of higher derivatives of Riemann zeta function}
\author[Mithun Kumar Das]{Mithun Kumar Das}
\address{ ICTP - The Abdus Salam International Centre for Theoretical Physics\\
Strada Costiera, 11, I - 34151, Trieste, Italy, \ \ and \ \
 School of Mathematical Sciences, National Institute of Science Education and Research, Bhimpur-Padanpur, Khurda, Odisha-752050, India. }
\email{das.mithun3@gmail.com}

\author[Sudhir Pujahari]{Sudhir Pujahari}
\address{School of Mathematical Sciences, National Institute of Science Education and Research, Bhimpur-Padanpur, Khurda, Odisha-752050, India.
}
\email{spujahari@niser.ac.in/spujahari@gmail.com}
\subjclass[2000]{11Mxx}
\keywords{ Riemann zeta-function, derivatives of Riemann zeta function, Hardy's $Z$-function,  derivatives of Hardy $Z$-function, Dirichlet polynomials, Mean square.}

\begin{abstract}
In this article, we extend the result of Conrey \cite[Theorem 2]{Con} to shorter intervals for higher-order derivatives of the zeta function.  That is we study the mean value of the product of two finite order derivatives of the zeta function multiplied by a mollifier in short intervals. In this process, we obtain better mollifier length in some short intervals compared to the length of mollifier  implied by Conrey's result.  
These finer studies allow us to refine the error term of some classical results of Levinson and Montgomery~\cite{LM}, Ki and Lee~\cite{KL} on zero density estimates of $\zeta^{(k)}$. 

Further, we showed that almost all non-trivial zeros of Matsumoto-Tanigawa's $\eta_k$-function cluster near the critical line. 
\end{abstract}

\maketitle
\section{Introduction}

\subsection{Zeros of derivatives of zeta functions}
Let $s=\sigma+it$ denote a complex variable, and $\zeta(s)$ be the Riemann zeta function, which satisfies the functional equation \[h(s)\zeta(s)=h(1-s)\zeta(1-s),\] where $h(s)=\pi^{-s/2}\Gamma(s/2)$. 
The famous Riemann hypothesis, predicted by Riemann in 1859, says that all the complex zeros of the zeta function lie on the critical line $\Re(s)=\frac{1}{2}.$ From the work of Speiser~\cite{Speiser}, we know that the Riemann hypothesis is equivalent to the  fact that the derivative of the zeta function does not have zeros in the strip $0< \Re(s)<\frac{1}{2}$.  The distribution of zeros of $\zeta^{(k)}$ has been extensively studied in the literature (see~\cite{Berndt},~\cite{Spira1}~\cite{Spira2},~\cite{Spira3}~\cite{Spira4}).   

 In 1974, Levinson and Montgomery~\cite{LM} conducted an extensive examination of the distribution of zeros pertaining to the higher derivatives of the zeta function. Specifically, they established that if $\rho_k = \beta_k + i \gamma_k$ represents the non-trivial zeros of $\zeta^{(k)}, k\geq 1$ and $\Li(x)=\int_{2}^x(\log{t})^{-1}dt$, then, under the assumption of the Riemann hypothesis, the following expression holds: 
  \begin{align}\label{urh}
\sum_{\substack{0 \leq \gamma_k \leq T\\ \beta_k> \frac{1}{2}  }} \Big(\beta_k - \frac{1}{2} \Big) = k\frac{T}{2\pi}\log\log{\frac{T}{2\pi}} + \frac{T\log{2}}{4\pi}- \frac{T}{2\pi}k\log\log{2} - k\Li\Big(\frac{T}{2\pi}\Big) + \O(\log{T}).
\end{align}
Furthermore, they established that $\zeta^{(k)}(s)$, where $k\geq 1$, possesses only a finite number of complex zeros in the region $\Re(s)<\frac{1}{2}$. Unconditionally they demonstrated that for $T^a\leq U \leq T$, $a >\frac{1}{2}$,   
\begin{align}\label{lm}
\sum_{\substack{T \leq \gamma_1\leq T+U\\ \beta_1> \frac{1}{2}  }} \Big(\beta_1 - \frac{1}{2} \Big) = \frac{U}{2\pi}\log\log{\frac{T}{2\pi}} +\O(U) \mbox{ and } \sum_{\substack{T \leq \gamma_1\leq T+U\\ \beta_1< \frac{1}{2}  }} \Big(\frac{1}{2}-\beta_1  \Big) = \O(U). 
\end{align}
Recently,  Ki and Lee~\cite{KL} extended the above unconditional result \eqref{lm}. Precisely, they showed that for
  $T^a\leq U \leq T$, $a >\frac{1}{2}$, 
\begin{align}\label{kl}
\sum_{\substack{T \leq \gamma_k\leq T+U\\ \beta_k> \frac{1}{2}  }} \Big(\beta_k - \frac{1}{2} \Big) = k\frac{U}{2\pi}\log\log{\frac{T}{2\pi}} +\O(U) \mbox{ and } \sum_{\substack{T \leq \gamma_k\leq T+U\\ \beta_k< \frac{1}{2}  }} \Big(\frac{1}{2}-\beta_k  \Big) = \O(U).
\end{align}

 In~\cite[Theorem 1]{CG5}, Conrey and Ghosh have shown that all most all zeros of $\zeta^{k}(s)$ are in the region $\sigma> \frac{1}{2}-\frac{\phi(t)}{\log \, t}$ for any $\phi(t) \to \infty$ as $t \to \infty$.  Additionally, they proved that, for any $c>0$, $\zeta^{k}(s)$ exhibits a positive proportion of zeros in the region $\sigma \geq \frac{1}{2}+ \frac{c}{\log \, t}.$  Furthermore, assuming the Riemann hypothesis, they demonstrated that, for any $\epsilon>0$, $\zeta^{k}$ possesses $\gg_{\epsilon } T$ zeros in the region \[\frac{1}{2} \leq \sigma < \frac{1}{2}+ \frac{(1+\epsilon) \log \, \log \, T}{\log \, T}, \ \ 0<t<T.\]

In this article, we obtain the explicit constant in the error terms of Levinson and Montgomery's results ~\cite[Theorem 5, Corollary of Theorem 4]{LM} as well as  Ki and Lee's results ~\cite[Theorem 3 and Theorem 2]{KL}. 
More explicitly, we refine the error terms appearing in \eqref{lm} and \eqref{kl} as follows:
 \begin{theorem}\label{thm6}
Let $ 0 < \theta < \frac{2k+1}{4(k+1)}$ and  $\rho_k = \beta_k + i \gamma_k$ be zero of $\zeta^{(k)}$ such that $T \leq \gamma_k \leq T+H $ for an integer $k\geq 1$. Then, for any $H=T^a, \frac{1}{2} +\theta < a \leq 1$ we have           
  \begin{align*}
(a)\quad \sum_{\substack{T \leq \gamma_k \leq T+H \\ \beta_k>\frac{1}{2}}} \left(\beta_k - \frac{1}{2}\right) \leq k\frac{H}{2\pi}\log\log{\frac{T}{2\pi }} &
+ \frac{H}{4\pi}\log{\Big(1+ \frac{2}{(2k+1)\theta} + \frac{2k^2\theta}{3(2k-1)}\Big)}\\
\nonumber & -\frac{H}{2\pi}k\log\log{2} + \O\Big(\frac{H(\log\log{T})^3}{\log{T}}\Big).
 \end{align*}
 \begin{align*}
(b)\quad \sum_{\substack{T \leq \gamma_k \leq T+H \\ \beta_k<\frac{1}{2}}} \left(\frac{1}{2} - \beta_k \right) \leq  \frac{H}{4\pi}\log{\Big(\frac{1}{2}+ \frac{1}{(2k+1)\theta} + \frac{k^2\theta}{3(2k-1)}\Big)} + \O\Big(\frac{H(\log\log{T})^3}{\log{T}}\Big).
 \end{align*}
\end{theorem}
In the following corollary of Theorem~\ref{thm6} we obtain the optimal value of the second order term of $H$.           
 \begin{corollary}\label{coro6}
 Let $k\geq 4$ and $H=T^a$ where $\frac{1}{2} +\sqrt{\frac{3(2k-1)}{k^2(2k+1)}} < a \leq 1$. Then, we have 
 \begin{align*}
\sum_{\substack{T \leq \gamma_k \leq T+H \\ \beta_k> \frac{1}{2}  }} \Big(\beta_k - \frac{1}{2} \Big) \leq k\frac{H}{2\pi}\log\log{\frac{T}{2\pi}} &+\frac{H}{4\pi}\log{\Big(1+\frac{4k}{\sqrt{12k^2-3}}\Big)}\\
&-k\frac{H}{2\pi}\log\log{2}+\O\Big(\frac{H(\log\log{T})^3}{\log{T}}\Big).
\end{align*}

 \begin{align*}
&\sum_{\substack{T \leq \gamma_k \leq T+H \\ \beta_k< \frac{1}{2}  }} \Big(\frac{1}{2} - \beta_k  \Big) \leq \frac{H}{4\pi}\log{\Big(\frac{1}{2}+\frac{2k}{\sqrt{12k^2-3}}\Big)}+ \O\Big(\frac{H(\log\log{T})^3}{\log{T}}\Big).
\end{align*}
 \end{corollary}      
 
\begin{figure}
	\includegraphics[scale=0.50]{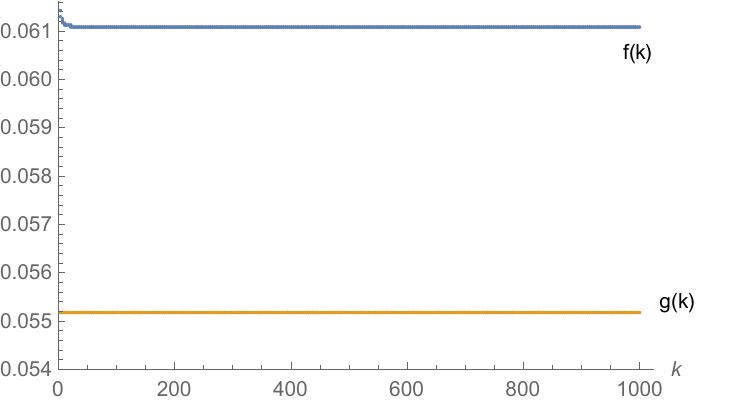}
	\caption{Plot for $f(k)=\frac{1}{4\pi}\log{\Big(1+\frac{4k}{\sqrt{12k^2-3}}\Big)}$ and $g(k)=\frac{1}{4\pi}\log{2}$, $4\leq k\leq 1000$.}\label{figure 1}
\end{figure} 
\subsubsection*{\bf Comparison between Theorem~\ref{thm6} and the conditional result \eqref{urh}}
The comparison between the second-order main term in the conditional result \eqref{urh} of Levinson and Montgomery, and the corresponding term in Theorem~\ref{thm6}(a), reveals a distinction (see Figure 1). 
We have $\log{\big(1+\frac{2}{(2k+1)\theta} + \frac{2k^2\theta}{3(2k-1)}\big)}$ in place of $\log{2}$ in the second order main term of \eqref{urh}.  
 Note that for $k\geq 1$ and $\theta>0,$ the optimal value of $\frac{2}{(2k+1)\theta} + \frac{2k^2\theta}{3(2k-1)}$ is $\frac{4k}{\sqrt{12k^2-3}}$.
This optimal value is attained when $\theta$ lies in the interval $(0,\frac{2k+1}{4(k+1)})$, slightly extending to the right of $\frac{2k+1}{4(k+1)}$ for $k \geq 4$, and for $k=1,2,3$, respectively.
 Note that $\frac{4k}{\sqrt{12k^2-3}}$ exhibits a decreasing trend. For $k\geq 4$ the values fall within the range of $(1.1547, \, 1.1639)$ while for $k=1,2,3$, the values are $\frac{4}{3}, 1.1925\ldots,$ and $ 1.1710\ldots$, respectively. 
 
  From this we can conclude that improving the range of $\theta$ in Corollary~\ref{coro5} with the mollifier would not improve the second order term in Corollary~\ref{coro6}. 
 
 Figure~\ref{figure 1} compare the conditional bound of the coefficient of $T$(2nd term only), i.e., $g(k):=\frac{1}{4\pi}\log{2}$ given in \eqref{urh} with the unconditional optimal bound of the coefficient of $H$ (2nd term only), i.e.,$f(k):=\frac{1}{4\pi}\log{\Big(1+\frac{4k}{\sqrt{12k^2-3}}\Big)}$  from the first part of Corollary~\ref{coro6}.

\subsubsection*{\bf Methodology}
In~\cite{LM}, Levinson and Montgomery have used the well-known Littlewood lemma on the function $ z_k(s):= (-1)^k2^s(\log{2})^{-k}\zeta^{(k)}(s)$ and  concluded their results. Ki and Lee~\cite{KL} applied the Littlewood lemma on a function which is a linear combination of derivatives of $\zeta$ and by using an upper bound estimate of a fractional moment of logarithmic derivatives of $\zeta$, proved their unconditional result. In our case, we consider the function $z_k(s)\Phi(s)$ in the place of $z_k(s)$, where $\Phi(s)$ is a mollifier and apply Littlewood lemma, this refines the explicit constant of the error terms. The main technique of our method is the asymptotic estimate of the mean square of the product of finite order derivatives of the Riemann zeta function with mollifier within short intervals. This mean square result obtain here in connection with the technique developed in our previous paper \cite{DP1}, which motivated form the work of Selberg~\cite{Selb} and Levinson~\cite{Levi}.

\subsection{Mean square of the product of $\zeta^{(k)}$ with mollifier}
 In 1989, Conrey~\cite{Con} proved that more than two-fifths of the zeros of the Riemann zeta function are on the critical line. To achieve this he showed the following result. 
Let $B(s):= \sum_{n\leq X=T^{\theta}}\frac{b(n)}{n^{s+R/(\log \, T)}}$,  where $b(n)= \mu(n)P\Big(1-\frac{\log{n}}{\log{T^{\theta}}}\Big)$ with $0<\theta< 4/7$,  $0<R \ll 1$ and $P$ be a polynomial such that $P(0)=0, P(1)=1$.    For any polynomial $Q$, consider $V(s):=Q\Big(-\frac{1}{\log \, T}\frac{d}{ds}\Big)\zeta{(s)}$.  Then,
 \begin{align}\label{C1}
 \int_{T}^{2T}\Big|VB\Big(\frac{1}{2}-\frac{R}{\log \, T}+it\Big)\Big|^2 dt \sim c(P,Q,R)T \,\,\mbox{ as } T\rightarrow \infty,
 \end{align}
 where
 \begin{align*}
 c(P,Q,R)= |P(1)Q(0)|^2+ \frac{1}{\theta}\int_{0}^{1}\int_{0}^{1}e^{2Ry}|Q(y)P'(x)+\theta Q'(y)P(x)+ \theta RQ(y)P(x)|^2
dxdy. \end{align*}
 The above mean square result is the key ingredient to prove his result.

In~\cite[(7)]{CG5}, Conrey and Ghosh have stated that for $B(s)$ defined as above with $R=0$, $\theta=\frac{1}{2} $ and there exist a $b(n)$ such that
\begin{equation}\label{ECG1}
\int_2^{T} \left| \frac{\zeta^{(k)}(\frac{1}{2}+it)}{\log^k t} B\left(\frac{1}{2}+it \right)\right|^2dt \sim \left(\frac{1}{2} + \frac{k}{\sqrt{4k^2-1}}\coth\left(\frac{k}{2}\sqrt{\frac{2k+1}{2k-1}}\right)\right)T.
\end{equation}
In ~\cite[Theorem 1]{CGG},  Conrey, Ghosh and Gonek studied similar results in the interval $[1, T]$ with product of two mollfiers in more general settings for $0<\theta < \frac{1}{2}$.  
Our interest here is mainly for $P(x)=x$ and $Q(x)=x^k, k \geq 0$ and $R=0$ in \eqref{C1}. Moreover, we consider product of higher order derivatives of zeta functions. In a personal communication, Conrey mentioned that the mean square result of \eqref{C1} can be extended to $R=0$.

We extend \cite[Theorem 2]{Con} to short intervals $[T, T+H]$ for a shorter length of mollifier by walking through his estimates carefully. 
\begin{theorem}\label{Thm1}
 Let $H=T^a$ and $X=T^{\theta}$. For $ 0 < \theta < \frac{2(7a-5)}{7}$ and $\frac{5}{7}\leq a\leq 1$, we have 
 \begin{align*}
& \int_{T}^{T+H}\Big|\frac{\zeta^{(k)}\big(\frac{1}{2}+it\big)}{\log^k{T}}\sum_{n\leq T^{\theta}}\frac{\mu{(n)}}{n^{\frac{1}{2}+it}}\Big(1- \frac{\log{n}}{\theta\log{T}}\Big)\Big|^2dt \sim
 \Big(\frac{1}{2}+ \frac{1}{(2k+1)\theta}+ \frac{\theta k^2}{3(2k-1)} \Big)H.
 \end{align*}
\end{theorem}

A sketch of the proof of Theorem \ref{Thm1} is given in Section \ref{proofs}.

 In the next result, we extend the range of short intervals up to $\frac{1}{2}< a\leq 1$.  Moreover, we improve the length of the mollifier under certain conditions (see Remark \ref{remark 1} and Figure \ref{figure 2} for detailed discussions) on $a$, compared to Theorem \ref{Thm1}.
 Further, we consider the product of two higher order derivatives of $\zeta$-function.
 \begin{theorem}\label{thm5}
 Let $m, n$ be any non-negative integer and $k=\min\{m,n\}$.  Then, for any $ H=T^a, \frac{1}{2} +\theta < a \leq 1$, where $X= T^\theta, 0<\theta < \frac{2k+1}{4(k+1)}$ we have
\begin{align*}
&\int_{T}^{T+H}\zeta^{(m)}\Big(\frac{1}{2}+it \Big)\zeta^{(n)}\Big(\frac{1}{2}-it \Big) \left|\sum_{n\leq T^{\theta}}\frac{\mu{(n)}}{n^{\frac{1}{2}+it}}\Big(1- \frac{\log{n}}{\theta\log{T}}\Big)\right|^2 dt\\
 &= (-1)^{m+n}\Big[\frac{1}{2} +  \frac{1}{\theta (m+n+1)} + \frac{mn\theta}{3(m+n-1)} + \O\Big(\frac{(\log\log{T})^{3}}{\log{T}}\Big)\Big]H \Big(\log{\frac{T}{2\pi}} \Big)^{m+n}.
\end{align*}  
 \end{theorem}

 \begin{remark}
In Theorem \ref{thm5}, we have given the upper bound of $\theta$ as $\frac{2k+1}{4(k+1)}$ for $k=\min\{m,n\}$ and $m, n$ are non-negative integers. We use Theorem~\ref{thm5} to prove Theorem~\ref{thm6} and when $k\geq 4$, the range of $\theta$ here is sufficient to get optimal bound in Theorem~\ref{thm6}.
 \end{remark}

 In particular, choosing $m=n=k$ we obtain the mean value of  higher derivatives of Riemann zeta function product with the Dirichlet polynomial in the following corollary: 
 \begin{corollary}\label{coro5}
  Let  $ 0 < \theta < \frac{2k+1}{4(k+1)}$. Then, for any $H=T^a$ and  $\frac{1}{2} + \theta < a \leq 1$   we have 
 \begin{align*}
& \int_{T}^{T+H}\Big|\zeta^{(k)}\Big(\frac{1}{2}+it\Big)\sum_{n\leq T^{\theta}}\frac{\mu{(n)}}{n^{\frac{1}{2}+it}}\Big(1- \frac{\log{n}}{\theta\log{T}}\Big)\Big|^2dt \\
 & = \Big(\frac{1}{2}+ \frac{1}{(2k+1)\theta}+ \frac{\theta k^2}{3(2k-1)} + \O\Big(\frac{(\log\log{T})^{3}}{\log{T}}\Big)\Big)H\Big(\log{\frac{T}{2\pi}}\Big)^{2k}.
 \end{align*}
  \end{corollary}

\begin{remark}\label{remark 1}
\begin{figure}\label{figure 2}
	\includegraphics[scale=0.7]{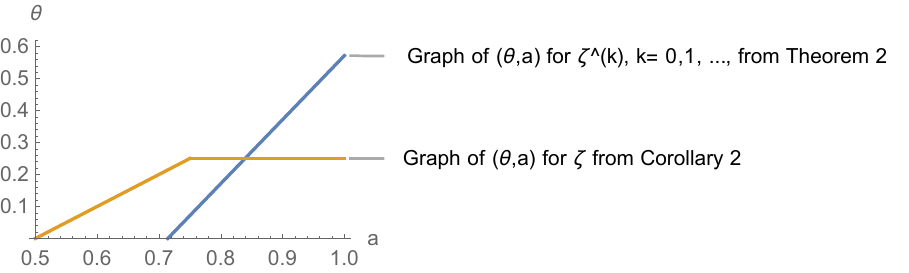}
	\includegraphics[scale=0.7]{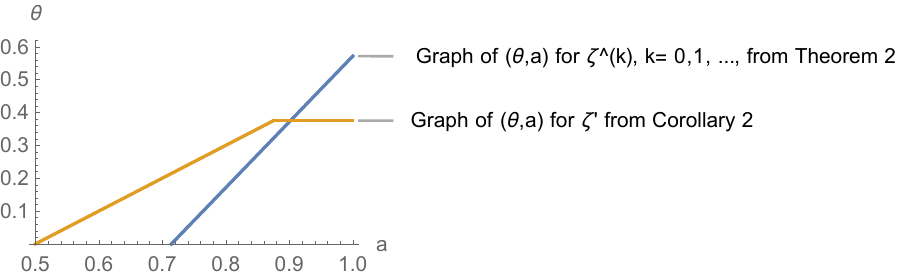}
	\includegraphics[scale=0.7]{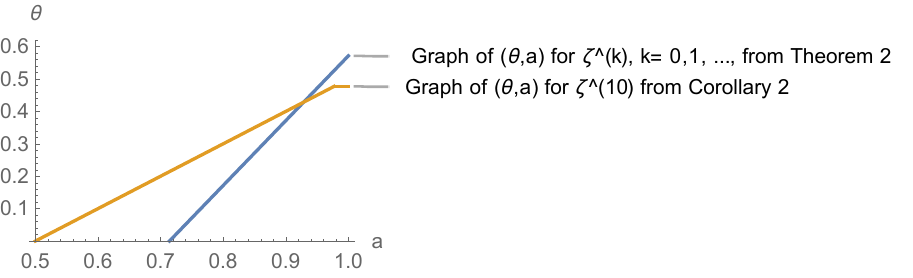}
	\caption{Comparison between the plots of $(a,\theta)$ for Theorem \ref{Thm1} and Corollary \ref{coro5} }\label{figure 2}
\end{figure} 
 Corollary ~\ref{coro5} improves significantly the permissible ranges of $\theta$ and $a$ in Theorem \ref{Thm1}.  That is, we have $\theta \leq \min\{a - \frac{1}{2}, \frac{2k+1}{4(k+1)}\}$ with $\frac{1}{2} <a\leq 1$ in Corollary \ref{coro5} where as  $\theta < \frac{2(7a-5)}{7}$ with $\frac{5}{7}\leq a\leq 1$ in Theorem \ref{Thm1}. We have pictorially demonstrate the comparison in Figure \ref{figure 2} above.
 
 For instance, if $k=0$ and $a=\frac{3}{4}$, in Corollary ~\ref{coro5} we have $\theta < \frac{1}{4}$, where as in Theorem \ref{Thm1} we have $\theta < \frac{1}{14}$. 
  Moreover, for $\frac{3}{4}<a \leq 1$,  $\theta < \frac{1}{4}$ in Corollary~\ref{coro5}.  For $\frac{1}{2}< a \leq .84$  the range  $\theta$ is better in Corollary \ref{coro5} compared to Theorem \ref{Thm1}. In contrast, the range of $\theta$ in Theorem \ref{Thm1} is better than the range of $\theta$ in Corollary \ref{coro5} for $0.84\leq a\leq 1$ (see the first graph in Figure \ref{figure 2} ).

 For $k>2$ we get a larger range of $a$ i.e., $a \in (\frac{1}{2}, \frac{13}{14} ]$, in which the range of $\theta$ in Theorem \ref{thm5} and Corollary \ref{coro5} is bigger than the range of $\theta$ in Theorem \ref{Thm1}.  In the other hand for $a \in (\frac{13}{14}, 1 ]$, the range of $\theta$ in Theorem \ref{thm5} is smaller than the range of $\theta$ in Theorem \ref{Thm1} (see the 3rd graph in Figure \ref{figure 2}  ). 
\end{remark}

 \subsection{Zeros of Matsumoto-Tanigawa's $\eta_k$-function} 
 The Hardy's $Z$-function $Z(t)$ is a real valued function for $ t \in \mathbb{R}$ which is defined by
 \begin{align}\label{hardyz}
 Z(t)= e^{i\theta(t)}\zeta(1/2 + it),
 \end{align}
 where $\theta(t)=\arg{h(1/2+it)}$. In other words,
\begin{equation*}
    Z(t) = \zeta\left(\frac{1}{2}+it\right)\chi\left(\frac{1}{2}+it\right)^{-\frac{1}{2}},
\end{equation*}
where $\chi(s) = \frac{h(1-s)}{h(s)}.$ Clearly, $e^{i\theta(t)}=\chi\left(\frac{1}{2}+it\right)^{-\frac{1}{2}}.$ Also, the following definitions play important roles to describe the results of this subsection.  In this subsection we follow the notations from~\cite{Ande} and ~\cite{MT}.  That is let
 
 \begin{equation}\label{omega}
 \omega(s):=\frac{\chi'}{\chi}(s),
 \end{equation}
 \begin{equation}\label{lambda}
 \lambda(s):=\dfrac{\omega'}{\omega}(s)- \frac{1}{2}\omega(s).
 \end{equation}
For any positive integer $k$,
\begin{align}\label{int2}
\eta_k(s):=\lambda_k(s)\zeta(s)+\sum_{j=1}^{k-1}\binom{k}{j}\lambda_{k-j}(s)\zeta^{(j)}(s)-\frac{2}{\omega(s)}\zeta^{(k)}(s),
\end{align}
where $\lambda_1(s)=1, \lambda_2(s)=\lambda(s)$ and for $k\geq 1$ 
\begin{align}\label{lambda-recu}
\lambda_{k+1}(s)=\lambda(s)\lambda_k(s)+ \lambda_k^{'}(s).
\end{align}
 In order to study the zeros of $Z^{'}(t)$, Anderson~\cite{Ande} defined a meromorphic function $\eta(s)$ 
\begin{align}\label{eta}
\eta(s)= \zeta(s)-\frac{\zeta'(s)}{\omega(s)},
\end{align}
where $\omega$ is as defined in \eqref{omega}.  He proved that the function $\eta(s)$ satisfies the relation 
\begin{align*}
Z'(t)=\frac{d}{dt}\left(e^{i\theta(t)}\right) \eta\left(\frac{1}{2}+it\right)
\end{align*}
and that all the zeros of $Z'(t)$ are the zeros of $\eta$ of the form $\frac{1}{2}+it$. He also counted the number of zeros of $\eta$ in some vertical strip as well as on the critical line and under Riemann hypothesis he proved that all the non-trivial zeros of $\eta$-function are along the critical line.   
The  higher derivatives case was studied by Matsumoto and Tanigawa~\cite{MT}.   To study the zeros of $Z^{(k)}(t)$
they introduced a meromorphic function $\eta_k$ as defined in \eqref{int2} and proved the following;  \\
{\bf (a)}
\begin{equation*}
 \eta_{k+1}(s)=\lambda(s)\eta_k(s)+ \eta'_k(s), \quad k\geq 1,
 \end{equation*} 
 with $\lambda(s)$ is as in \eqref{lambda} and $\eta_1(s) :=\eta(s)$.
 
{\bf (b)}  The function $\eta_k$ satisfy the functional equation
 \begin{align}\label{eta-equation}
 \eta_k(s) = (-1)^k\chi(s)\eta_k(1-s).
 \end{align}  

{\bf (c)}
Let 
$R_k(T) :=  \lbrace \sigma+it: 1-2m < \sigma < 2m, 0 <t <T \rbrace $
and 
$N_{R_k}(T)$ be the number of zeros of $\eta_k$ in the rectangle $R_k.$   Then 
\begin{eqnarray}\label{eta_k_zero}
  N_{R_k}(T)=  \frac{T}{2\pi}\log{\frac{T}{2\pi}} - \frac{T}{2\pi} + \O_k(\log{T}).
 \end{eqnarray}
 
 {\bf (d)} For $k\geq 1$
 \begin{align}\label{int1}
 Z^{(k)}(t)=i^{k-1}\frac{d}{dt}\left(e^{i\theta(t)}\right) \eta_k\left(\frac{1}{2}+it \right).
 \end{align}

 {\bf (e)}
 If the number of zeros of $Z^{(k)}$ in $(0, T)$ is $N_{k,0}(T)$,  then
 \begin{align*}
 T\log{T} \ll N_{k,0}(T) \leq  \frac{T}{2\pi}\log{\frac{T}{2\pi}} - \frac{T}{2\pi} + \O_k(\log{T}).
 \end{align*} 
 
 \begin{remark}
 One has the identity  $\omega\left(\frac{1}{2}+it\right)= -2\theta'(t)$ and by using it we extend the recurrence relation of $\eta_k$ for all non-negative integers by defining $\eta_0(s)=\frac{2\zeta(s)}{\omega(s)}$.  This gives the extension of the relation \eqref{int1} to all non-negative integers. 
 \end{remark}
 

In the next result we investigate the distribution of zeros of $\eta_k$ near the critical line.  To state the result for any positive integer $m$, we define the following region
$$U_{m}(H):=\{\sigma+it: \frac{1}{2} <\sigma < m, T<t \leq T+H \ \mbox{ with} \  2< m=H^{\kappa}, 0< \kappa<1 \}$$
 and
  $$N_{\eta_k}( m, H):= |\{ \rho_k= \beta_k +i\gamma_k \in U_m(H) : \eta_k(\rho_k)=0  \}|.$$

\begin{theorem}\label{thm2}
 Let $k$ be a positive integer, $ 0 < \theta < \frac{2k+1}{4(k+1)}$ and $H=T^a, \frac{1}{2} + \theta < a \leq 1$.  Let $ \rho_k:= \beta_k+i\gamma_k$ denotes the zeros of $\eta_k(s)$ in the region $U_{m}(H)$.   As $T \to \infty$,
 \begin{align*}
 \Big(\beta_k-\frac{1}{2}\Big)N_{\eta_k}(m, H) \leq \sum_{\substack{\eta_k(\rho_k)=0\\\frac{1}{2}<\beta_k\leq m}} \left(\beta_k- \frac{1}{2}\right) \leq \left(\frac{\log{(4^kP_k(\theta))}}{2} + \o(1) \right)\frac{H}{2\pi}.
 \end{align*}
 \end{theorem}
%
%
From \eqref{eta_k_zero} and Theorem~\ref{thm2} we conclude that all most all zeros of $\eta_k$-function are near the critical line.

 Conray and Ghosh~\cite[p. 195]{CG} remarked that for each positive integer $k$,  there exist a meromorphic function $Z_k(s)$ satisfying the  functional equation $Z_k(s)=(-1)^k\chi(s)Z_k(1-s)$ such that $|Z^{(k)}(t)|=\left|Z_k\left(\frac{1}{2}+it\right)\right|$.
 In the following remark, we are providing an example of such functions. 
 \begin{remark}
Let $\omega$ as in \eqref{omega} and $\eta_k$ as in \eqref{int2} satisfy \eqref{eta-equation}.  If we consider $Z_k(s):= \frac{\omega \eta_k}{2i^k}(s)$ for $k\geq 0$. Then $Z_k$ is a meromorphic function, satisfying the conditions remarked by Conrey and Ghosh~\cite[p. 195]{CG}. 
 \end{remark}
 
 The article is organized as follows.   In Section 2,  we recall and prove some preliminary results that we need to prove our main results. In the final section, we prove the main results.
     
 
 \section{Preliminaries}\label{pre}
 Now, recall the second part of Proposition~2 from \cite{DP1}, which is one of the main ingredient to prove Theorem~\ref{thm5} and we state it as a lemma. To serve our purpose, we define  
\[G_k(t):=\frac{Z^{(k)}{(t)}}{(\log{\tau})^{k}} \quad \mbox{ and }\quad J_{k_1, k_2}(T):= \int_{T}^{T+H}{G_{k_1}(t)G_{k_2}(t) \Big|\sum_{n\leq T^{\theta}}\frac{\mu{(n)}}{n^{\frac{1}{2}+it}}\Big(1- \frac{\log{n}}{\theta\log{T}}\Big)\Big|^2}dt,\] where $\tau= \sqrt{t/2\pi}$.
\begin{lemma}[Proposition~2, \cite{DP1}]\label{L12}
Let $k_1, k_2$ be non-negative integers such that $\kappa:=\min(k_1, k_2)$ and define $\vartheta(k_1,k_2)= (-1)^{\frac{k_1-k_2}{2}}$ or $0$ according to $k_1+k_2$ is even or odd. Also, let $0< \theta < \frac{2\kappa+1}{4(\kappa+1)} $ and $H= T^a, \frac{1}{2}+\theta < a <\frac{4\kappa+3}{4(\kappa+1)}$. For $k_1, k_2 \geq 1 $, we have  
\begin{align}\label{asym-sp}
J_{k_1, k_2}(T)=  \vartheta(k_1,k_2)\left(1 + \frac{1}{\theta (k_1+k_2+1)} + \frac{4 k_1k_2 \theta }{3(k_1+k_2-1)} + \O\Big(\frac{(\log\log{T})^3}{\log{T}}\Big)\right)H.
\end{align}
Further, we have
\begin{align}\label{ext-sp}
J_{k_1, k_2}(T) = \begin{cases}
 H\Big(1+ \frac{1}{\theta} + \O\Big(\frac{(\log\log{T})^3}{\log{T}}\Big)\Big), \mbox{ if } k_1 = 0,  k_2 =0,\\
 \O\Big(\frac{H(\log\log{T})^3}{\log{T}}\Big) , \mbox{ if }  k_1 = 0 , k_2 =1,\\
\vartheta{(0,k_2)} H\Big(1 +\frac{1}{(k_2+1)\theta} +\O\Big(\frac{(\log\log{T})^3}{\log{T}}\Big)\Big), \mbox{ if }  k_1 = 0 , k_2 \geq 2.
\end{cases}
\end{align}
\end{lemma}
 We want to obtain the upper bound of higher derivatives of zeta function and Hardy's $Z$-function. For this we need the following essential tool.
     
\begin{lemma}\label{L15}
Let $\mathscr{D}:= \left\lbrace s: \frac{1}{2}\leq\sigma\leq m(t),~ 2\pi \leq t \right\rbrace$ with $m(t)=\o(t)$. For $k\geq 1$, $\tau = \sqrt{\frac{t}{2\pi}}$ and $\lambda_k$ is as in \eqref{lambda-recu} we have
 $$\lambda_k(s)=\left(\log{\tau}\right)^{k-1} + \O\left(t^{-1}(\log{t})^{k-1}\right).$$
\end{lemma}
\begin{proof}
First, we expand the recurrence relation~\eqref{lambda-recu} of $\lambda_k$ for $k\geq 3$. By using induction principle we write 
\begin{align*}
\lambda_k(s)=\lambda(s)^{k-1} + \Lambda_k{(s)}  \mbox{ with } \Lambda_k(s)=\sum_{m=0}^{k-3}(\lambda(s))^{k-3-m}P_{m+1}(\lambda)(s)
\end{align*}
and 
\begin{align*}
P_n(\lambda):=\sum_{(m_1,m_2,...,m_n)}a_{(m_1,m_2,...,m_n)}\prod_{j=1}^n \left(\lambda^{(j)}\right)^{m_j},
\end{align*}
where $a_{(m_1,m_2,...,m_n)}$ are non-negative integers for $0\leq m_j\leq n-1$ such that $\sum_{j=1}^njm_j\leq  n$. 
 By logarithmic derivative of $\chi(s)$ along with~\cite[Eq: 3.1, 3.2]{MT} we write 
\begin{align*}
\omega(s)= \frac{\pi}{2}\tan{\frac{\pi s}{2}} - \log{\frac{s}{2\pi}} +\frac{1}{2s}+\O\left(\frac{1}{|s|^2}\right),
\end{align*}
where $\sigma>0$ and $\omega(s)$ is as in \eqref{omega}. Using the fact that for $s\in \mathscr{D}$, $\tan{\frac{\pi s}{2}} = i + \O(e^{-\pi t})$ we obtain
\begin{align}\label{omegabound}
\omega(s)= - \log{\frac{|s|}{2\pi}} + \O\left(\frac{1}{|s|}\right) \mbox{ and } \omega^{'}{(s)} = \frac{\pi^2}{4}\sec^2{\frac{\pi s}{2}} + \O\left(\frac{1}{|s|}\right)= \O\left(\frac{1}{|s|}\right).
\end{align}
Using \eqref{lambda} we deduce
$$\lambda(s)= \log{\tau} + \O\left(\frac{1}{t}\right).$$ 
Taking the higher derivatives of $\lambda(s)$ in \eqref{lambda} we have 
\begin{align*}
\lambda^{(k)}(s)= \left(\frac{\omega'}{\omega}\right)^{(k)}(s) - \frac{1}{2}\omega^{(k)}(s).
\end{align*}
For $s \in \mathscr{D}$ we obtain
\begin{align*}
\lambda^{(k)}(s) = \O{(t^{-k}}).
\end{align*}
 Hence, for $k\geq 1$ and $s\in \mathscr{D}$ we get
  $$\lambda_k(s)=\left(\log{\tau}\right)^{k-1} + \O\left(t^{-1}(\log{t})^{k-2}\right).$$
  This completes the proof of the lemma.
\end{proof}

\begin{lemma}\label{L14}
For any $\sigma\geq \frac{1}{2}$, $t\geq 2\pi$ and non-negative integer $k$ we have
\begin{align*}
\zeta^{(k)}(s) \ll \left(t^{\frac{1-\sigma}{3}}+1\right)(\log{t})^{k+1}.
\end{align*}
 Further, we get
\begin{align*}
|Z^{(k)}(t)| \ll t^{\frac{1}{6}}(\log{t})^{k+1}.
\end{align*}
\end{lemma}

\begin{proof}
To prove the first part we consider the well known bound of $\zeta$-function in the range $\sigma \geq \frac{1}{2}$ and $t\geq 2\pi$ (see \cite[eq: 2.4]{GI}), i.e.,
  $$\zeta(s)\ll \left(t^{\frac{1-\sigma}{3}}+1\right)\log{t}.$$
  
   By Cauchy integral formula we write
\begin{align*}
\zeta^{(k)}(s)= \frac{k!}{2\pi i}\int_{|z-s|=\rho}\frac{\zeta(z)}{(z-s)^{k+1}}dz,
\end{align*}
where $\rho$ is suitably chosen small radius of a circle centred at $s$. Then 
\begin{align*}
\zeta^{(k)}(s) \ll \frac{\left(t^{\frac{1-\sigma}{3}}+1\right)\log{t}}{\rho^k}.
\end{align*}
Taking $\rho = \frac{1}{\log{t}}$ to get the required result. 

\textbf{2nd part}:
Recall $\eta_k$ from \eqref{int2} and $\lambda_k$ from \eqref{lambda-recu}. By Lemma~\ref{L15}, for $2\pi\leq t \leq 2T$, we have \[\lambda_k(1/2 + it)=\O\left((\log{\tau})^{k-1}\right).\]  Combining above bounds with \eqref{int2} and \eqref{omegabound} we obtain the result.
\end{proof}
 
 Next, we establish the relation between higher derivatives of zeta function and Hardy's $Z$-function. For the first derivative, Hall~\cite{Hall} deduced the following identity:
 \begin{align}\label{hall-identity}
 \zeta^{'}\Big( \frac{1}{2}+it\Big)= e^{-i\theta(t)}\{Z^{'}(t) -i\theta^{'}(t)Z(t)\}.
 \end{align}
 Thus, 
 \begin{align*}
 \left|\zeta^{'}\Big( \frac{1}{2}+it\Big)\right|^2= Z^{'}(t)^2 + \theta^{'}(t)^2Z(t)^2.
 \end{align*}
 From identity \eqref{hall-identity} we obtain identities for higher derivative case upto some error term in the following lemma.
 \begin{lemma}\label{darivative-zeta}
 Let $n\geq 2$ and $\tau= \sqrt{\frac{t}{2 \pi}}$  with $T\leq t \leq 2T$. We get 
 \begin{align*}
 i^{n-1}\zeta^{(n)}\Big( \frac{1}{2}+it\Big)= e^{-i\theta(t)} \sum_{j=0}^{n-1}\binom{n-1}{j}&\Big((-i\log{\tau})^{n-1-j} Z^{(j+1)}{(t)} + (-i\log{\tau})^{n-j}Z^{(j)}(t)\Big)\\
\nonumber & +\O\Big(T^{-\frac{5}{6}}\log^{n-1}{T}\Big).
 \end{align*} 
 \end{lemma}
 \begin{proof} Setting $f(t):= Z^{'}(t) -i\theta^{'}(t)Z(t)$ and applying Leibniz's rule on \eqref{hall-identity} we deduce
 \begin{align}\label{zeta^n}
 i^{n-1}\zeta^{(n)}\Big( \frac{1}{2}+it\Big)= \sum_{j=0}^{n-1}\binom{n-1}{j}\frac{d^{j}}{dt^{j}}f(t)  \frac{d^{n-1-j}}{dt^{n-1-j}}e^{-i\theta(t)}.
 \end{align}
 From \cite[eq: 30]{Hall} we have $\theta^{'}(t)= \log{\tau} + \O(t^{-1})$ and $\theta^{(k)}(t)= \O(t^{-k+1})$ for $k\geq 2$. So, for a positive integer $k$ we get
 \begin{align}\label{e^n}
 \frac{d^{k}}{dt^{k}} e^{-i\theta(t)} = (-i \log{\tau})^{k}e^{-i\theta(t)} + \O\Big(\frac{(\log{t})^{k-1}}{t}\Big).
 \end{align}
 Again by Leibniz's rule
 \begin{align*}
 \frac{d^{k}}{dt^{k}}f(t) = Z^{(k+1)}(t)-i \sum_{l=0}^{k}\binom{k}{l}\frac{d^{l}}{dt^{l}}\theta^{'}(t) Z^{(k-l)}(t)
 = Z^{(k+1)}(t)- i  Z^{(k)}(t) \log{\tau} +\O\Big(\frac{|Z^{(k-1)}(t)|}{t}\Big).
  \end{align*}
 Thus, applying second part of Lemma~\ref{L14} to the above error term we get
\begin{align}\label{f^k}
\frac{d^{k}}{dt^{k}}f(t) = Z^{(k+1)}(t)- i  Z^{(k)}(t) \log{\tau} +\O\Big(\frac{\log^k{t}}{t^{\frac{5}{6}}}\Big).
\end{align}
Hence, using \eqref{e^n} and \eqref{f^k} in \eqref{zeta^n} we conclude the formula.
 \end{proof}

\section{Proof of theorems}\label{proofs}
 \begin{proof}[\textbf{Proof of Theorem \ref{Thm1}}]
   To prove \eqref{C1}, which is essentially Theorem 2 of \cite{Con}, Conrey improved the bound for $\xi_0$ in \cite[Theorme 2]{BCH} for $\zeta$-function case.  More explicitly, he improves $ \xi_0\ll \Delta^{-\frac{7}{2}}T^{\frac{5}{2}+\epsilon}X^{\frac{17}{12}}T^{\frac{1}{4}}$ in  (C1) of \cite[Theorme 2]{BCH} to $\xi_0\ll \Delta^{-\frac{7}{2}}T^{\frac{5}{2}+2\epsilon}X^{3\epsilon} (T^{\frac{1}{2}}X^{\frac{7}{8}}+ X^{\frac{7}{4}})$ (see eq: (75) in \cite{Con}), where $\epsilon>0$ is sufficiently small, which allows $\theta < \frac{4}{7}$ in Theorem 2 of \cite{Con}  and hence in Theorem 1 of \cite{BCH}. 
   
 Note that a modification of calculations in page 179--180 of \cite{BCH} gives
\begin{align*}
&\int_{T}^{T+H}\Big|\zeta\Big(\frac{1}{2}+it\Big)\sum_{n\leq T^{\theta}}\frac{\mu{(n)}}{n^{\frac{1}{2}+it}}\Big(1- \frac{\log{n}}{\theta\log{T}}\Big)\Big|^2dt\\
&= \int_{T}^{T+H}\sum_{q,l\leq T^{\theta}}\frac{b_l\overline{b_q}}{lq}(l,q)\left(\log{\frac{t(l,q)^2}{2\pi lq}}  + 2\gamma \right) dt + \O\left(\xi_0(H+\Delta \log{T}) + \frac{\Delta^2}{T}\log^3{T}\right),
\end{align*}
where $\exp{(5\sqrt{\log{T}})} \leq \Delta < H(\log{T})^{-1}$, $H=T^a, a\leq 1$ and $b_n:= \mu(n)\big(1- \frac{\log{n}}{\theta\log{T}}\big)$. The integral term has been already computed in \cite{BCH} (also derived in Proposition 2 of \cite{DP1}), which is asymptotic to $H\big(1+\frac{1}{\theta}\big)$ for any $H,\theta>0$.  Along with improved bound of $\xi_0$ by Conrey \cite{Con}, the estimate of the above error term is
\begin{align*}
\ll \Delta^{-\frac{7}{2}}T^{\frac{5}{2}+2\epsilon}X^{3\epsilon} (T^{\frac{1}{2}}X^{\frac{7}{8}}+ X^{\frac{7}{4}})(H+\Delta \log{T}) + \frac{\Delta^2}{T}\log^3{T}.
\end{align*}

For $H=T^a$ and $\Delta=H^{1-\epsilon}$, the last error term is much smaller than $H$ and  the first error term is 
\begin{align*}
\ll T^{\frac{5}{2}(1-a)+ 6\epsilon}X^{3\epsilon} (T^{\frac{1}{2}}X^{\frac{7}{8}}+ X^{\frac{7}{4}})\ll T^{3-\frac{5}{2}a+\frac{7}{8}\theta +8\epsilon} + T^{\frac{5}{2}(1-a) + \frac{7}{4}\theta +8\epsilon}\ll H^{1-\delta}, 
\end{align*}
for sufficiently small $\delta>0$, provided that $\theta < \min\{\frac{4(7a-6)}{7}, \frac{2(7a-5)}{7} \}= \frac{2(7a-5)}{7}$ for $a\in [5/7, 1]$.

Thus, for  $\theta < \frac{2(7a-5)}{7}$ with $1\geq  a\geq \frac{5}{7}$, we can establish 
\begin{align*}
\int_{T}^{T+H}\Big|\zeta\Big(\frac{1}{2}+it\Big)\sum_{n\leq T^{\theta}}\frac{\mu{(n)}}{n^{\frac{1}{2}+it}}\Big(1- \frac{\log{n}}{\theta\log{T}}\Big)\Big|^2dt \sim H\Big(1+\frac{1}{\theta}\Big) \mbox{ as } T\rightarrow \infty.
\end{align*}
Further, in \cite{Con} Conrey considered the mean values of higher derivatives of $\zeta$-function. In this scenario,
he adopted the same strategy as in \cite{BCH} (see the proof of Theorem 2 in \cite{Con}). In our case also following the same arguments as for $k=0$ case as discussed above, we extend the mean value result for higher order derivatives of zeta function in short intervals, provided that $\theta < \frac{2(7a-5)}{7}$ with $1\geq  a\geq \frac{5}{7}$. Note that for $P(x)=x$ and $Q(x)=x^k, k \geq 0$ and $R=0$, the constant in \eqref{C1} can be computed as
\begin{align*}
c(P, Q, R)= \frac{1}{2}+ \frac{1}{(2k+1)\theta}+ \frac{\theta k^2}{3(2k-1)}.
\end{align*}
Hence, the result follows.
 \end{proof}
 
 \begin{proof}[\textbf{Proof of Theorem} \ref{thm5}]
  From Lemma~\ref{darivative-zeta} we get
 \begin{align}\label{zeta^n-final2}
(-i)^{n-1}\zeta^{(n)}\Big( \frac{1}{2}-it\Big)= e^{i\theta(t)} \sum_{j=0}^{n-1}\binom{n-1}{j}&\Big((i\log{\tau})^{n-1-j} Z^{(j+1)}{(t)} + (i\log{\tau})^{n-j}Z^{(j)}(t)\Big)\\
\nonumber & +\O\Big(T^{-\frac{5}{6}}\log^{n-1}{T}\Big).
\end{align}
Combining  Lemma~\ref{darivative-zeta} and \eqref{zeta^n-final2}, we have
\begin{small}
\begin{align}\label{zeta-prod1}
\zeta^{(m)}\Big(\frac{1}{2}+it \Big)\zeta^{(n)}\Big(\frac{1}{2}-it \Big)=(-\log{\tau})^{m+n}
&\sum_{k=0}^{n-1}\sum_{j=0}^{m-1}\binom{m-1}{j}\binom{n-1}{k}\frac{(-i)^ki^j}{(\log{\tau})^{j+k+2}}A(j,k)\\
\nonumber & + \O\left(T^{-\frac{2}{3}}(\log{T})^{m+n-2}\right),
\end{align}
\end{small}
where $A(j,k)$ is given by the expression 
\begin{align*}
A(j,k)& =  Z^{(j+1)}(t)Z^{(k+1)}(t) +i \log{\tau} Z^{(j+1)}(t)Z^{(k)}(t) -i\log{\tau} Z^{(j)}(t)Z^{(k+1)}(t) \\
&+(\log{\tau})^{2} Z^{(j)}(t)Z^{(k)}(t).
\end{align*}

Rewriting the above equation in terms of $G_k(t)$ that has been defined in Lemma~\ref{L12},  we have 
\begin{align*}
\frac{A(j,k)}{(\log{\tau})^{j+k+2}} = G_{j+1}(t)G_{k+1}(t) + iG_{j+1}(t)G_{k}(t)-iG_{j}(t)G_{k+1}(t) + G_{j}(t)G_{k}(t).
\end{align*}
 Now, applying Lemma~\ref{L12}, we evaluate the integral
$$\mathcal{I}(j,k, T):= \int\limits_{T}^{T+H}\frac{A(j,k)}{(\log{\tau})^{j+k+2}} \left|\sum_{n\leq T^{\theta}}\frac{\mu{(n)}}{n^{\frac{1}{2}+it}}\Big(1- \frac{\log{n}}{\theta\log{T}}\Big)\right|^2 dt. $$
We evaluate $\mathcal{I}(j,k, T)$ in two cases according to $j+k$ is even or odd.
For even case, applying Lemma~\ref{L12}, we get
 \begin{align}\label{even}
\mathcal{I}(j,k, T) &= i^{k-j}\Big[2 +\frac{1}{\theta}\Big(\frac{1}{k+j+3}+ \frac{1}{k+j+1}\Big)+ \frac{4\theta}{3}\Big(\frac{(j+1)(k+1)}{k+j+1}+ \frac{jk}{k+j-1}\Big)\Big]H\\
& + \O\Big(H\frac{(\log\log{T})^3}{\log{T}}\Big).
\end{align}
 For the case $j+k$ odd we get
\begin{align}\label{28/8/23:00}
\mathcal{I}(j,k, T) = i^k(-i)^j\Big[2 +\frac{2}{\theta}\Big( \frac{1}{k+j+2}\Big)+ \frac{4\theta}{3}\Big(\frac{(j+1)k}{k+j}+ \frac{j(k+1)}{k+j}\Big) + \O\Big(\frac{(\log\log{T})^3}{\log{T}}\Big)\Big]H.
\end{align}
Note that \eqref{even} and \eqref{28/8/23:00} are valid for $H= T^a, \frac{1}{2}+\theta < a <\frac{4\kappa+3}{4(\kappa+1)}$, where $0< \theta < \frac{2\kappa+1}{4(\kappa+1)} $ and $\kappa=\min \{j,k\}$. But we extend the range of $a$ to $ \frac{1}{2}+\theta < a \leq 1$ in the following way:  whenever $\frac{4\kappa+3}{4(\kappa+1)} \leq a \leq 1$ we can split the interval $[T, T+H]$ into subinterval of length $T^b$, where $ b <\frac{4\kappa+3}{4(\kappa+1)}$, for which \eqref{28/8/23:00} valid. Then adding the corresponding result we obtain that  \eqref{28/8/23:00} is true for $\frac{1}{2}+\theta < a \leq 1$.
 
  We have to evaluate the sum $$\sum_{k=0}^{n-1}\sum_{j=0}^{m-1}\binom{m-1}{j}\binom{n-1}{k}(-i)^ki^j\mathcal{I}(j,k, T).$$
Now, this sum can be break in two part as
\begin{align*}
\sum_{k=0}^{n-1}\sum_{\substack{j=0}}^{m-1}=\sum_{k=0}^{n-1}\sum_{\substack{j=0\\j+k \mbox{ even}}}^{m-1} +\sum_{k=0}^{n-1}\sum_{\substack{j=0\\j+k \mbox{ odd}}}^{m-1}.
\end{align*}
Thus, 
\begin{small}
\begin{align}\label{expression1}
\sum_{k=0}^{n-1}\sum_{j=0}^{m-1}\binom{m-1}{j}\binom{n-1}{k}(-i)^ki^j\mathcal{I}(j,k, T) &= 2H\sum_{k=0}^{n-1}\sum_{\substack{j=0}}^{m-1}\binom{m-1}{j}\binom{n-1}{k} +  \frac{H}{\theta}\mathcal{I}_1\\
&\nonumber + \frac{4H\theta}{3}\mathcal{I}_2 + \O\Big(\frac{H(\log\log{T})^3}{\log{T}}\Big),
\end{align}
\end{small}
where $\mathcal{I}_1, \mathcal{I}_2$ are defined by
\begin{footnotesize}
\begin{align*}
\mathcal{I}_1 := \sum_{k=0}^{n-1}\sum_{\substack{j=0\\j+k \mbox{ even}}}^{m-1}\binom{m-1}{j}\binom{n-1}{k}\Big(\frac{1}{k+j+3}+ \frac{1}{k+j+1}\Big) + \sum_{k=0}^{n-1}\sum_{\substack{j=0\\j+k \mbox{ odd}}}^{m-1}\binom{m-1}{j}\binom{n-1}{k}\frac{2}{k+j+2},
\end{align*}
\end{footnotesize}
and 
\begin{small}
\begin{align*}
\mathcal{I}_2& := \sum_{k=0}^{n-1}\sum_{\substack{j=0\\j+k \mbox{ even}}}^{m-1}\binom{m-1}{j}\binom{n-1}{k}\Big(\frac{(j+1)(k+1)}{k+j+1}+ \frac{jk}{k+j-1}\Big)\\
& + \sum_{k=0}^{n-1}\sum_{\substack{j=0\\j+k \mbox{ odd}}}^{m-1}\binom{m-1}{j}\binom{n-1}{k}\Big(\frac{(j+1)k}{k+j}+\frac{(k+1)j}{k+j}\Big).
\end{align*}
\end{small}
For a real variable $x$ we have the following binomial identity
\begin{align*}
(1+x)^{m+n-2} = \sum_{k=0}^{n-1}\sum_{\substack{j=0}}^{m-1}\binom{m-1}{j}\binom{n-1}{k}x^{j+k}.
\end{align*}
Multiplying both sides of above identity by $(1+x)^2$ and integrate in the interval $[-1, 1]$ we get
\begin{align}\label{I_1}
\frac{2^{m+n}}{m+n+1}= \frac{1}{2}\int_{-1}^{1}(x^2 + 1 + 2x)(1+x)^{m+n-2}dx = \mathcal{I}_1 .
\end{align}
Let us define real valued functions $f$ and $g$ by
\begin{align*}
f(x):= \frac{d}{dx}(x(1+x)^{m-1}) \frac{d}{dx} (x(1+x)^{n-1}) +\frac{d}{dx}((1+x)^{m-1})\frac{d}{dx}((1+x)^{n-1}),\\
g(x):= \frac{d}{dx}(x(1+x)^{m-1})\frac{d}{dx}((1+x)^{n-1}) + \frac{d}{dx}((1+x)^{m-1})\frac{d}{dx}(x(1+x)^{n-1}).
\end{align*}
Now, rewriting $f(x)$ and $g(x)$ in term of binomial formula and after differentiating we integrate the identities in the interval $[-1, 1]$ to get
\begin{align*}
\frac{1}{2} \int_{-1}^{1} \big(f(x)+g(x) \big) dx = \mathcal{I}_2.
\end{align*}
By the change of variable $z=1+x$ on the left hand side of above identity 
\begin{align}\label{I_2}
 \mathcal{I}_2= \frac{1}{2}mn\int_{0}^{2}z^{m+n-2}dz = \frac{mn}{m+n-1}2^{m+n-2}.
\end{align}

Thus combining \eqref{I_1} and \eqref{I_2} with \eqref{expression1} we get
\begin{align}\label{mean-}
\sum_{k=0}^{n-1}\sum_{j=0}^{m-1}&\binom{m-1}{j}\binom{n-1}{k}(-i)^ki^j\mathcal{I}(j,k, T) \\
\nonumber &= H2^{m+n}\Big(\frac{1}{2} +  \frac{1}{\theta (m+n+1)} + \frac{mn\theta}{3(m+n-1)} +\O\Big(\frac{(\log\log{T})^3}{\log{T}}\Big)\Big).
\end{align}
Hence, multiplying \eqref{zeta-prod1} by the absolute square of the mollifier and then by using \eqref{mean-} we deduce the result.
\end{proof}
 \begin{proof}[\textbf{Proof of Theorem} \ref{thm6}]
 \textbf{(a)}
 Let us set
\begin{align}\label{molli}
\Phi(s)= \sum_{n\leq T^{\theta}}\frac{\mu{(n)}}{n^{s}}\Big(1- \frac{\log{n}}{\theta\log{T}}\Big).
\end{align}
We would apply the well known Littlewood's lemma on $z_k(s)\Phi(s)$ in the region which is given by joining the vertices  $\frac{1}{2}+iT, 2+iT, 2+ i(T+H)$ and $\frac{1}{2}+i(T+H)$. The function $z_k(s)$ was introduced by Levinson and Montgomery~\cite[Section 3]{LM} and it is given by
 \begin{align*}
 z_k(s)= (-1)^k2^s(\log{2})^{-k}\zeta^{(k)}(s).
 \end{align*}
 Then the lemma gives
 \begin{align*}
2\pi\sum_{\substack{T \leq \gamma_k\leq T+H \\ \beta_k > \frac{1}{2}}} \Big(\beta_k - \frac{1}{2} \Big) &=\int_{T}^{T+H}\log\Big|z_k\Phi\Big(\frac{1}{2}+it\Big)\Big|dt - \int_{T}^{T+H}\log\Big|z_k\Phi\Big(2+it\Big)\Big|dt\\
& -\int_{\frac{1}{2}}^{2}\arg z_k\Phi{(\sigma+iT)}d\sigma + \int_{\frac{1}{2}}^{2}\arg z_k\Phi{(\sigma+i(T+H))}d\sigma.
\end{align*}
 
 By writing 
 $\Phi(s)=1+\Phi_1(s)$ and using the fact that $b_n \leq 1$ we have
 $|\Phi_1(s)|< 2^{1-\sigma} <1$ for $\sigma >3/2$. Hence $\log{\Phi(\sigma+it)}$ is analytic for $\sigma >3/2$. Also we have $|z_k(\sigma+it) -1|\leq \frac{1}{2}(2/3)^{\sigma/2}$ (see \cite[eq: 3.2]{LM}). Thus, $z_k\Phi{(\sigma+it)}=1+\O\big((2/3)^{\sigma/2}\big)$. By using Cauchy integral formula for the contour 
 $$\{\sigma+it : 2 \leq \sigma < \infty, T\leq t \leq T+H \},$$ we get 
\begin{align}\label{pfeq8}
\int_{T}^{T+H}\log{z_k\Phi{(2+it)}}dt =\O(1).
\end{align}
 
Since $\arg z_k\Phi=\arg z_k+\arg\Phi$, combining known results for the integral of $\arg z_k$ in $[\frac{1}{2}, 2]$ from \cite[Section 3]{LM} and   the estimate $\arg \Phi(\sigma+iT)=\O(\log{T})$ for $\sigma \geq 1/2 $ (see \cite[page 134]{Selb} ), we see that two horizontal line integral contribute at most $\O(\log{T})$.
 Thus, we get
\begin{align}\label{pfeq71}
2\pi\sum_{\substack{T \leq \gamma_k\leq T+H \\ \beta_k > \frac{1}{2}}} \Big(\beta_k - \frac{1}{2} \Big) & = \int_{T}^{T+H}\log\Big|\zeta^{(k)}\Big(\frac{1}{2}+it\Big)\Phi\Big(\frac{1}{2}+it\Big)\Big|dt )\\
\nonumber &+ H\Big(\frac{1}{2}\log{2} -k\log\log{2}\Big) + \O(\log{T}).
\end{align}
Now, by applying the arithmetic-geometric mean inequality on the integral appearing on the right-hand side of \eqref{pfeq71}, we have
\begin{align*}
\int_{T}^{T+H}\log\Big|\zeta^{(k)}\Big(\frac{1}{2}+it\Big)\Phi\Big(\frac{1}{2}+it\Big)\Big|dt  \leq \frac{H}{2}\log\Big(\frac{1}{H}\int_{T}^{T+H}\Big|\zeta^{(k)}\Big(\frac{1}{2}+it\Big)\Phi\Big(\frac{1}{2}+it\Big)\Big|^2dt\Big).
\end{align*}
 Hence, by using Corollary~\ref{coro5} on the right hand side of the above inequality we conclude the result
 \begin{align}\label{pfeq712}
2\pi\sum_{\substack{T \leq \gamma_k\leq T+H \\ \beta_k > \frac{1}{2}}} \Big(\beta_k - \frac{1}{2} \Big) \leq & kH\log\log{\frac{T}{2\pi }} + \frac{H}{2}\log{\Big(\frac{1}{2}+ \frac{1}{(2k+1)\theta} + \frac{k^2\theta}{3(2k-1)} \Big)}\\
\nonumber & + H\Big(\frac{\log{2}}{2}-k\log\log{2}\Big) + \O\Big(\frac{H(\log\log{T})^3}{\log{T}}\Big).
\end{align}
\textbf{(b)}\\
Levinson and Montgomery~\cite[Theorem 3]{LM} proved that for $H > 0,$
\begin{small}
\begin{align}\label{pfeq713}
\sum_{\substack{T \leq \gamma_k\leq T+H }} \Big(\beta_k - \frac{1}{2} \Big) = 
\frac{kH}{2\pi}\log\log{\frac{T}{2\pi}} + \frac{H}{2\pi}\Big(\frac{\log{2}}{2}-k\log\log{2}\Big) + \O\Big(\frac{H^2}{T\log{T}}+\log{T}\Big).
\end{align}
\end{small}
 Also, we have
 \begin{align}\label{pfeq714}
 -\sum_{\substack{T \leq \gamma_k\leq T+H }} \Big(\beta_k - \frac{1}{2} \Big) = \sum_{\substack{T \leq \gamma_k\leq T+H\\ \beta_k < 1/2 }} \Big(\frac{1}{2} -\beta_k\Big) - \sum_{\substack{T \leq \gamma_k\leq T+H \\ \beta_k > 1/2}} \Big(\beta_k - \frac{1}{2} \Big).
 \end{align}
 Hence, using \eqref{pfeq712} and \eqref{pfeq713} in \eqref{pfeq714} under the condition $H=T^a, \frac{1}{2} + \theta < a \leq 1$, where $0< \theta < \frac{2k+1}{4(k+1)}$, we obtain the inequality (b).
 \end{proof}   
 \begin{proof}[\textbf{Proof of Corollary} \ref{coro6}] We want to optimize the coefficient of $H$ in (a) and (b) of Theorem~\ref{thm6}. Clearly,  given a integer $k$ it is enough to obtain minimum value of the function
 \begin{align*}
 g(\theta):= \frac{1}{(2k+1)\theta} + \frac{k^2\theta}{3(2k-1)}, ~~~ 0<\theta < \frac{2k+1}{4(k+1)}.
\end{align*}
 We see that $g$ attains its minimum value at $\theta = \sqrt{\frac{3(2k-1)}{k^2(2k+1)}}$. But this extremal point must be smaller than $ \frac{2k+1}{4(k+1)}$. This condition gives $k \geq 4$.
  \end{proof}
  
\begin{proof}[\textbf{Proof of Theorem} \ref{thm2}] Let us set
\begin{align*}
f(s):=\eta_k(s) \Phi(s)\Big(-\frac{2}{\omega{(s)}}\Big)^{k-1},
\end{align*}
where  $\eta_k(s)$ as in \eqref{int2},  $\Phi(s)$ as in \eqref{molli} and $\omega(s)$ as in \eqref{omega}. Similar to Theorem~\ref{thm6} again we apply Littlewood lemma on $f(s)$
 in the rectangle $\mathcal{D}$ with vertices $\frac{1}{2}+iT, m+ iT, m+i(T+H)$ and $\frac{1}{2}+i(T+H)$,  where $m$ is a positive integer. 

 Note that $\omega(s)$ is non-vanishing and analytic on $\mathcal{D}$ (see \eqref{omegabound}).  Thus, the number of zeros of $\eta_k$ in $\mathcal{D}$ is less than or equal to the number of zeros of $f$ in $\mathcal{D}$.

 For $\sigma >1$, by using Lemma \ref{L15} in \eqref{int2} we get
 \begin{align}\label{pfeq10}
 \eta_k(\sigma+it)(\omega(\sigma+it))^{1-k}=\left(\frac{-1}{2}\right)^{k-1}\zeta(\sigma+it) + \O\left(\frac{1}{\log{t}}\right).
 \end{align}
 Thus,
 \begin{align}\label{pfeq9}
\int_{T}^{T+H}\log|f(m+it)|dt = \int_{T}^{T+H}\log|\zeta \Phi(m+it)|dt  + \O{(H{(\log{T}})^{-1})}.
 \end{align}
 Now, following the arguments for \eqref{pfeq8} we can show that the integration of $\log{\zeta\Phi(m+it)}$ in $[T, T+H]$ is $\O(1)$.
Since $\Re{(\log{f(s)})}=\log{|f(s)|}= \frac{1}{2}(\log{f(s)} + \overline{\log{f(s)}})$,  we conclude that
\begin{align}\label{tail-apporx}
\int_{T}^{T+H}\log|f\left(m+it\right)|dt= \O{(H{(\log{T}})^{-1})}.
\end{align}

Again, taking into account the fact that
$$\arg{(f(s))}= \Im{(\log{f(s)})} =\Im{(\log{\eta_k(s)\omega(s)^{1-k}})} + \Im{(\log{\Phi(s)})},$$ 
and hence for $\sigma>\frac{3}{2}$ we have $\arg{(f(s))} \ll 1$. 
 From \eqref{pfeq10}, it is not hard to show that
\begin{align}\label{arg-apporx1}
\int_{\frac{3}{2}}^{m}\arg f\left(\sigma+i(T+H)\right)d\sigma - \int_{\frac{3}{2}}^{m}\arg f\left(\sigma+iT\right)d\sigma \ll m.
 \end{align}
Now, note that $|\Re{f(2+it)}|> (2-\zeta(2))^2 > 0$ and from Lemma \ref{L14} and \eqref{int2} we obtain 
 \begin{align*}
 f(\sigma + it) \ll X^{1- \sigma} \Big(t^{\frac{1- \sigma}{3}} +1\Big)\log{t}\quad  \mbox{ for } \frac{1}{2} \leq \sigma \leq 2.
 \end{align*}
  Thus, by using \cite[Lemma, p. 213]{Titc} we obtain 
  \begin{align*}
 \arg f\left(\sigma + i(T+H)\right), \arg f\left(\sigma + iT \right) \ll \log{T} +\log{X}.
 \end{align*}
 Hence, for a real number $\beta < 2$ closer to $2$, we get 
 \begin{align}\label{arg-aporx2}
\int_{\frac{1}{2}}^{\beta}\arg f\left(\sigma+i(T+H)\right)d\sigma - \int_{\frac{1}{2}}^{\beta}\arg f\left(\sigma+iT\right)d\sigma \ll \log{T}.
 \end{align}
 
Recall  $G_k$ from Section~\ref{pre}. By using the formula \eqref{omegabound} and Lemma~\ref{L14} on $f$ we have 
\begin{align}\label{f-apporx}
\Big|f\Big( \frac{1}{2}+it\Big)\Big|= |G_k(t)| + \O\big(t^{-\frac{5}{6}}\log{t}\big).
\end{align} 
Combining \eqref{tail-apporx}, \eqref{arg-apporx1}, \eqref{arg-aporx2} and  \eqref{f-apporx} with the Littlewood lemma to get
\begin{align}\label{pfeq7}
2\pi\sum_{\substack{\eta_k(\rho_k^{'})=0 \\\rho_k^{'} \in \mathcal{D}, \beta_k^{'}> \frac{1}{2} }} \left(\beta_k^{'} - \frac{1}{2} \right) & \leq \int_{T}^{T+H}\log{\Big|G_k(t)\Phi\left(\frac{1}{2}+it\right)\Big|}dt +\O\Big(\frac{H}{\log{T}}\Big).
\end{align}
Now, by applying the arithmetic-geometric mean inequality on the integral appearing on the right-hand side of \eqref{pfeq7} and then appealing Lemma~\ref{L12} to conclude the result for  $H=T^a$, where $\frac{1}{2}+\theta<a < \frac{4k+3}{4(k+1)}$. For $\frac{4k+3}{4(k+1)} \leq a \leq 1$, by proceeding in a similar way as in the paragraph after \eqref{28/8/23:00} to conclude the required result. 
 \end{proof}
 


 \section{Acknowledgement}
We would like to thanks Prof. J. B. Conrey for helpful comments and Prof.  R. Balasubramanian for many helpful discussions. The authors would like to thank the referee for reviewing the manuscript in detail and giving valuable comments.  We also would like to thank the referee for bringing \cite{CGG} to our attantion.
The first author is supported by DST, Government of India under the DST-INSPIRE Faculty Scheme with Faculty Reg. No. IFA21-MA 168. The second author is supported by Science and Engineering Research Board [SRG/2023/000930].


\begin{thebibliography}{}
\bibitem{Ande}
R. J. Anderson, {\it On the function Z(t) associated with the Riemann zeta-function}, J. Math.
Anal. Appl. 118 (1986), 323--340.

\bibitem{BCH}
R. Balasubramanian, J. B. Conrey, D. R. Heath-Brown, {\it
Asymptotic mean square of the product of the Riemann zeta-function and a Dirichlet polynomial.} J. Reine Angew. Math. 357 (1985), 161-181.

\bibitem{BCR} S. Bettin, V. Chandee, M. Radziwill, {\it The mean square of the product of the Riemann zeta-function with Dirichlet polynomials.} J. Reine Angew. Math. 729 (2017), 51--79.
\bibitem{Berndt}
B. C.  Berndt,  {\it The number of zeros for $\zeta^{(k)}(s)$},  J. London Math. Soc., 2 (1970), 577--580. 
 
\bibitem{Con}
 J. B. Conrey, {\it More than two fifths of the zeros of the Riemann zeta function are on the critical line,} J. Reine Angew. Math. 399 (1989), 1-26.
\bibitem{CGG}
J. B. Conrey; A. Ghosh; S. M. Gonek, {\it Mean values of the Riemann zeta-function with application to the distribution of zeros},
Number theory, trace formulas and discrete groups (Oslo, 1987), 185--99, Academic Press, Boston, MA, 1989.

 \bibitem{CG}
J. B. Conrey, A. Ghosh, {\it A mean value theorem for the Riemann zeta-function at its relative extrema on the critical line.} J. London Math. Soc.(2) 32 (1985), no. 2, 193-202.
 
 \bibitem{CG5}
 J.B. Conrey, A. Ghosh,  {\it Zeros of derivatives of the Riemann zeta-function near the critical line} Analytic number theory (Allerton Park, IL, 1989), 95--110, Progr. Math., 85, Birkhäuser Boston, Boston,  MA, 1990.
 
%
 \bibitem{DP1}
 M. K. Das, S. Pujahari, { \it Mean of the product of derivatives of Hardy's $Z$-function with Dirichlet Polynomial.} J. Number Theory 258 (2024), 334--367.
\bibitem{GI}
 S. M. Gonek, A. Ivi\'c, {\it On the distribution of positive and negative values of Hardy's $Z$-function.} J. Number Theory 174 (2017), 189-201.
\bibitem{Hall}
R. R. Hall, {\it The behaviour of the Riemann zeta-function on the critical line.}
Mathematika 46 (1999), no. 2, 281-313.


\bibitem{KL}
H. Ki, Y. Lee, {\it Zeros of the derivatives of the Riemann zeta-function,} Functiones et Approximatio 47 (2012), no. 1, 79--87.
\bibitem{Levi}
 N. Levinson, {\it More than one third of zeros of Riemann's zeta-function are on $\sigma=1/2$.} Advances in Math. 13 (1974), 383--436.
 \bibitem{LM}
 N. Levinson, H. Montgomery, {\it Zeros of the derivatives of the Riemann zeta-function,} Acta Mathematica 133, (1974), 49--65.
\bibitem{MT}
K. Matsumoto, Y. Tanigawa, {\it On the zeros of higher derivatives of Hardy's $Z$-function,} J. Number Theory 75 (1999) 262--278.


\bibitem{Speiser}
A. Speiser,  {\it Geometrisches zur Riemannschen Zetafunktion}, Math. Ann., 110 (1934),
514--521.

\bibitem{Spira1}
R. Spira,  {\it Zero-free regions of $\zeta^{(k)}(s)$},  J. London Math. Soc., 40 (1965), 677--682.

\bibitem{Spira2}
------,  {\it Another zero-free region for $\zeta^{(k)}(s)$},  Proc. Amer. Math. Soc.,  26 (1970),  246--247. 
\bibitem{Spira3}
------,  Zeros of $\zeta^{'}(s)$ in the critical strip, Proc. Amer. Math. Soc., 35 (1972), 59--60. 
\bibitem{Spira4}
------,  {\it Zeros of $\zeta^{'}(s)$ and the Riemann hypothesis.}, Illinois J. Math., 17 (1973), 147--152.

\bibitem{Selb}
A. Selberg, {\it On the zeros of Riemann's zeta-function,} Skr. Norske Vid. Akad. Oslo I (1942), no. 10, reprinted in Collected Papers, Vol. I, Springer-Verlag, Berlin, 1989, 85--141.
\bibitem{Titc}
E. C. Titchmarsh, {\it The Theory of Riemann Zeta-Function}, 2nd edition, Oxford University Press, Oxford, 1986. 
\end{thebibliography}
\end{document}